\documentclass[12pt]{svmult}
\usepackage{times}
\usepackage{mathptmx,helvet,courier,graphicx,makeidx,multicol,footmisc}
\usepackage{hyperref}
\usepackage{color}
\usepackage[T1]{fontenc}
\usepackage[utf8]{inputenc}
\usepackage[american]{babel}
\usepackage{cancel}
\usepackage{dsfont}
\usepackage[a4paper]{geometry}
\newcommand{\msc}[2][2000]{%
  \let\@oldtitle\@title%
  \gdef\@title{\@oldtitle\footnotetext{#1 \emph{Mathematics subject
        classification.} #2}}%
}

\def\R{{\bf R}}
\def\N{{\bf N}}

\def\({\left(}
\def\){\right)}
\def\<{\left\langle}
\def\>{\right\rangle}

\def\Tend#1#2{\mathop{\longrightarrow}\limits_{#1\rightarrow#2}}

\def\eps{\varepsilon}

\def\op{{\rm op}}

\newcommand{\eqref}[1]{(\ref{#1})}

\newcommand{\To}{\longrightarrow}

\begin{document}

\title*{Semiclassical analysis of dispersion phenomena}
\author{Victor Chabu, Clotilde~Fermanian-Kammerer, Fabricio Maci\`{a}}
\institute{V. Chabu \at Universidade de São Paulo, IF-USP, DFMA, CP 66.318 05314-970, São Paulo, SP, Brazil,
\email{vbchabu@if.usp.br},\\
C. Fermanian Kammerer \at LAMA, UMR CNRS 8050,
Universit\'e Paris Est,
61, avenue du G\'en\'eral de Gaulle
94010 Cr\'eteil Cedex, France, 
\email{Clotilde.Fermanian@u-pec.fr}\\
F. Macia \at Universidad Polit\'{e}cnica de Madrid. DCAIN, ETSI Navales. Avda. de la Memoria 4, 28040 Madrid, Spain,
\email{fabricio.macia@upm.es}}

\maketitle

{\bf Abstract} : {\it Our aim in this work is to give some quantitative insight on the dispersive effects exhibited by solutions of a semiclassical Schr\"odinger-type equation in $\R^d$. We  describe quantitatively the localisation of the energy in  a long-time semiclassical limit within this non compact geometry and exhibit conditions under which the energy remains localized on compact sets. We also explain how our results can be applied in a straightforward way to describe obstructions to the validity of smoothing type estimates. }

\section{Introduction}
\subsection{Description of the problem}
Our aim in this work is to revisit some of the results obtained in \cite{CFM} in order to give some quantitative insight on the dispersive effects exhibited by solutions of the semiclassical Schrödinger-type  equation: 
\begin{equation}\label{eq:scdisp}
\left\{ \begin{array}{l}
i\eps \partial_{t} v^{\eps}(t,x) = \lambda(\eps D_x) v^{\eps}(t,x) + \eps^2 V(x) v^{\eps}(t,x) , \quad (t,x)\in\R\times\R^d, \vspace{0.2cm}\\
v^{\eps}|_{t=0}= u_{0}^{\eps},
\end{array}\right.
\end{equation}
Above, $\lambda,V \in \mathcal{C}^\infty (\R^d;  \R)$; the function $\lambda$ is the symbol of the semiclassical Fourier multiplyier defined by:
$$\forall v \in L^2(\R^d),\quad\lambda(\eps D_x)v(x):=\int_{\R^d}\lambda(\eps \xi)\widehat{v}(\xi)\mathrm{e}^{i\xi\cdot x}\frac{d\xi}{(2\pi)^d},$$
where, in general, the integral has to be understood in distributional sense. The following convention for the Fourier transform is used:
$$\widehat{v}(\xi):=\int_{\R^d}v(x)\mathrm{e}^{-i\xi\cdot x}d\xi.$$
Our goal is to understand the behavior as $\eps\to 0^+$ of solutions to \eqref{eq:scdisp} corresponding to sequences of initial data $(u^\eps_0)$ whose characteristic length-scale of oscillations is of order at least $\eps$ (see \eqref{def:xoscclassical} for a precise definition) at very long times, of the order of $1/\eps$.

To this aim, we scale in time the solutions to \eqref{eq:scdisp} and define:
$$u^\eps(t,\cdot):=v^\eps\left(\frac{t}{\eps},\cdot \right).$$
Note that these functions solve the following problem.
\begin{equation}\label{eq:disp}
\left\{ \begin{array}{l}
i\eps^2 \partial_{t} u^{\eps}(t,x) = \lambda(\eps D_x) u^{\eps}(t,x) + \eps^2 V(x) u^{\eps}(t,x) , \quad (t,x)\in\R\times\R^d, \vspace{0.2cm}\\
u^{\eps}|_{t=0}= u_{0}^{\eps},
\end{array}\right.
\end{equation}
If the symbol $\lambda$ happens to be homogeneous of degree two, \eqref{eq:disp} reduces to the non-semiclassical equation:
$$i \partial_{t} u^{\eps}(t,x) = \lambda(D_x) u^{\eps}(t,x) +V(x) u^{\eps}(t,x).$$
In what follows we shall consider sequences of initial data $(u^\eps_0)$ that are bounded in $L^2(\R^d)$. Denote by $(u^\eps)$ the corresponding sequence of solutions to \eqref{eq:disp} and construct the position densities:
$$n^\eps(t,x):=|u^\eps(t,x)|^2.$$
For every $t\in\R$, the sequence $(n^\eps(t,\cdot))$ is bounded in $L^1(\R^d)$, since 
$$||n^\eps(t,\cdot)||_{L^1(\R^d)}=||u^\eps(t,\cdot)||_{L^2(\R^d)}=||u^\eps_0||_{L^2(\R^d)};$$
 it is not difficult to show, using the fact that $u^\eps$ solve \eqref{eq:disp}, that there exists a subsequence $\eps_n\to 0^+$ and a $t$-measurable family of finite positive Radon measures $\nu_t(dx)$ on $\R^d$ such that
the space-time averages of the position densities $(n^{\eps_n})$ converge:
\begin{equation}\label{eq:density}
\lim_{n\to\infty}\int_a^b \int_{\R^d} \phi(x) |u^{\eps_n}(t,x)|^2 dx dt=\int_a^b \int_{\R^d} \phi(x)\nu_t(dx)dt,
\end{equation}
for every $a<b$ and every $\phi\in{\mathcal C}_0(\R^d)$. The limiting measure $\nu_t$ is sometimes called a \textit{defect measure} of the sequence $(u^\eps)$. It will follow from our results that defect measures give indeed a quantitative description of the lack of dispersion for solutions to \eqref{eq:scdisp}. 

\medskip

The long-time semiclassical limit has been studied with some detail in the context of Schrödinger equations in compact geometries, see for instance \cite{MaciaAv, MaciaTorus, AM:14, AFM:15}. In the compact setting, the dispersive nature of the equation manifests through more subtle mechanisms, and is intimately related to the global dynamics of the underlying classical system.

\medskip 

When the potential $V$ in \eqref{eq:disp} is identically equal to zero, simple calculations can be implemented for specific initial data. Construct for example, for  $\xi_0\in\R^d$, and $\theta\in {\mathcal S}(\R^d)$ with $\| \theta\|_{L^2(\R^d)}=1$: 
\begin{equation}\label{eq:data0}
u^\eps_{\xi_0}(x)=\theta(x) {\rm e}^{{i\over \eps} \xi_0\cdot x}.
\end{equation}
These sequences of initial data are highly oscillating as soon as $\xi_0\neq 0$ and therefore converge weakly to zero in $L^2(\R^d)$; however $|u^\eps_{\xi_0}|^2=|\theta|^2$ is independent of $\eps$ and $\xi_0$. 

A direct application of the stationary/non-stationary phase principle gives for any $\phi\in\mathcal{C}_c(\R^d)$ and any $a<b$ that the following limits hold. 
\begin{itemize}
\item If $\xi_0$ is not a critical point of $\lambda$, then 
$$\lim_{\eps\to 0^+}\int_a^b \int_{\R^d} \phi(x) |\mathrm{e}^{\frac{it}{\eps^2}\lambda(\eps D_x)}u^\eps_{\xi_0}(x)|^2 dx dt=0.$$
\item 
If $\xi_0$ is a critical point of $\lambda$, then 
$$\lim_{\eps\to 0^+}\int_a^b \int_{\R^d} \phi(x) |\mathrm{e}^{\frac{it}{\eps^2}\lambda(\eps D_x)}u^\eps_{\xi_0}(x)|^2 dx dt= \int_a^b\int_{\R^d} \phi(x) \left| {\rm e}^{-\frac{it}{2} \nabla^2 \lambda(\xi_0) D_x\cdot D_x} \theta(x)\right|^2 dxdt,$$
where $\nabla^2 \lambda(\xi_0)$ denotes the Hessian of $\lambda$ at the point $\xi_0$.
\end{itemize}

When $\xi_0$ is not a critical point of the symbol $\lambda$, the fact that no energy remains on any compact set in the high frequency limit is precisely a manifestation of dispersive behavior of the semiclassical problem~\eqref{eq:scdisp}. However, when $\xi_0$ happens to be a critical point of $\lambda$, such a dispersive behavior fails, and a fraction of the energy remains localized on compact sets of $\R^d$. Note that $\xi_0=0$ plays a special role in this setting, since it corresponds to initial data that are not oscillating. Therefore, even if $\xi_0=0$ is a critical point of $\lambda$,\footnote{Think for instance of $\lambda(\xi) = \| \xi \|^2$, for which \eqref{eq:disp} corresponds to the standard, non-semiclassical, Schrödinger equation, one of the most studied dispersive equations.} the fact that the local energy does not escape from every compact set as $\eps\to 0^+$ in this case should not be interpreted as a lack of dispersion of \eqref{eq:scdisp}. 

\medskip

The situation can be more intricate for initial data which are superposition of oscillating functions of the form above:
\begin{equation}\label{eq:data0bis}
u^\eps_0(x)=\theta_1(x) {\rm e}^{{i\over \eps} \xi_1\cdot x}+\theta_2(x) {\rm e}^{{i\over \eps} \xi_2\cdot x}
\end{equation}
with $\theta_1,\theta_2\in {\mathcal S}(\R^d)$, both non-zero, and $\xi_1,\xi_2\in\R^d$ such that $\xi_2$ is a critical point of $\lambda$ while $\xi_1$ is not. One easily checks that:
\begin{equation}\label{eq:ex}
\lim_{\eps\to 0^+}\int_a^b \int_{\R^d} \phi(x) |u^\eps(t,x)|^2 dx dt= \int_a^b\int_{\R^d} \phi(x) \left| {\rm e}^{-\frac{it}{2} \nabla^2 \lambda(\xi_2) D_x\cdot D_x} \theta_2(x)\right|^2 dxdt,
\end{equation}
which shows that only a fraction of the (asymptotic) total mass $\| \theta_1\|^2_{L^2(\R^d)}+\| \theta_2\|^2_{L^2(\R^d)}$ of the sequence of solutions is dispersed in this case. 

\medskip

Our aim here, is to provide a general description of these high frequency effects. In particular, we will generalize the analysis done in the previous examples to arbitrary sequence of initial data, and investigate the effects produced by the presence of a bounded non-zero potential $V$. We will also show in Corollary~\ref{cor:smoothing} how our results can be applied in a straightforward way to describe obstructions to the validity of smoothing type estimates for equations of the form \eqref{eq:disp} in the presence of critical points of the symbol $\lambda$.

\subsection{Non dispersive effects associated to isolated critical points}

We are first going to show that, in the presence of isolated critical points of $\lambda$, some of the high frequency effects exhibited by the sequence of initial data persist after applying the time evolution~(\ref{eq:disp}). As we said before, we give a complete description of the asymptotic behavior of the densities $|u^\eps(t,x)|^2$ associated to a  sequence of solutions to \eqref{eq:disp} issued from a sequence of initial data $\left(u_{0}^{\eps}\right)_{\eps >0}$ bounded in $L^2(\R^d)$. When the critical points of the symbol $\lambda$ are non-degenerate, we present an explicit procedure to compute all weak-$\star$ accumulation points of the sequence of time-dependent positive measures $\left(|u^{\eps}(t,\cdot)|^2\right)_{\eps >0}$ in terms of quantities that only depend on the sequence of initial data.

\medskip 

In order to prevent that all the mass of the sequence $\left(|u^{\eps}(t,\cdot)|^2\right)_{\eps >0}$ trivially escapes to infinity, we must make sure that the characteristic length scale of the oscillations of the sequence of initial data is at least of order $\eps$. The following, now standard, assumption is sufficient for our purposes:

\begin{enumerate}
\item[\textbf{A0}] The family $(u^\eps_0)_{\eps > 0}$ is uniformly bounded in $L^2(\R^d)$ and $\eps$-oscillating, in the sense that its energy is concentrated on frequencies smaller or equal than $1/\eps$ :
\begin{equation}\label{def:xoscclassical}
\limsup_{\eps\rightarrow 0}\int_{\|\xi\|>R/\eps} | \widehat {u^\eps_0} (\xi) | ^2 d\xi \Tend{R}{+\infty} 0,
\end{equation} 
\end{enumerate} 

In order to keep the presentation relatively simple, we also impose smoothness and growth conditions on $\lambda$ and $V$. More precisely:

\begin{enumerate}
\item[\textbf{A1}] $V\in \mathcal{C}^\infty(\R^d)$ is bounded together with all its derivatives and that $\lambda\in \mathcal{C}^\infty(\R^d)$ is a symbol of order $N>0$ (as in~\cite{DimassiSjostrand}, definition~7.5):
$$\forall\alpha\in\N^d,\quad \sup_{\xi\in\R^d}|\partial_\xi^\alpha\lambda(\xi)|\left(1+\|\xi\|\right)^{-N}<\infty.$$
\end{enumerate}

Our last hypothesis deals with the set of critical points of the symbol: 
$$\Lambda := \{\xi \in \R^d : \nabla\lambda(\xi) = 0 \}.$$ 
In our first result, we assume the following.
\begin{enumerate}
\item[\textbf{A2}] $\Lambda$ is a countable set of $\R^d$. 
\end{enumerate} 
 
\begin{theorem}\label{theo:disc}
Assume that the sequence of initial data $\left(u^\eps_0\right)_{\eps > 0}$ verifies {\bf A0} and that $\lambda$ and $V$ satisfy {\bf A1} and \textbf{A2}; denote by~$\left(u^\eps\right)_{\eps > 0}$ the corresponding family of solutions to (\ref{eq:disp}). Suppose $(\eps_n)_{n\in\N}$ is a subsequence along which $(|u^{\eps_n}|^2)_{n\in\N}$ converges, in the sense of \eqref{eq:density}, to some defect measure $\nu_t(dx)dt$. Then, for almost every $t\in\R$ the following holds:
\begin{equation}\label{eq:limu}
\nu_t(dx)\geq \sum_{\xi\in \Lambda}|u_\xi (t,x)|^2dx,
\end{equation}
where $u_\xi$ is a solution to the following Schrödinger equation:
\begin{equation}\label{eq:schrohprofil}
\left\{\begin{array}{l}
i\partial_t u_\xi(t,x) =\frac{1}{2}\nabla^2\lambda(\xi)D_x\cdot D_x u_\xi(t,x) +V(x)u_\xi(t,x),\medskip\\
u_\xi|_{t=0}=u_\xi^0,
\end{array} \right.
\end{equation}
and $u_\xi^0$ is the limit, for the weak topology on $L^2(\R^d)$, of the sequence $({\rm e}^{-\frac{i}{\eps_n} \xi \cdot x} u^{\eps_n}_0)_{n\in\N}$. 

\noindent If in addition, all critical points of  $\lambda$ are non-degenerate, then inequality~(\ref{eq:limu}) becomes an equality.
\end{theorem}

We will show below (see Proposition \ref{prop:disp1}) that, when at least one of the critical points of $\lambda$ is degenerate, there exist sequences of initial data for which inequality \eqref{eq:limu} is strict. However, even when the non-degeneracy condition is violated, there are simple conditions on the sequence of initial data that ensure that \eqref{eq:limu} is an equality. In order to state those, let us consider a cut-off function $\chi\in{\mathcal C} _0^\infty (\R^d)$ such that 
\begin{equation}\label{def:chi}
0\leq \chi\leq 1, \;\;\chi(\eta)=1\;\;{\rm  for}\;\; \|\eta\|\leq 1 \;\;{\rm and}\;\;  \chi(\eta)=0\;\;{\rm  for}\;\; \|\eta\|\geq 2.
\end{equation} 

\begin{theorem}\label{theo:data}
Assume that the same hypotheses as in Theorem \ref{theo:disc} hold, and that the following additional condition on the sequence $\left(u^\eps_0\right)_{\eps > 0}$ of initial data is satisfied: for all $\xi\in\Lambda$,
$$\limsup_{\delta\rightarrow 0^+}\, \limsup_{R\rightarrow +\infty}\, \limsup_{\eps\rightarrow 0^+} \left\|\left(1-\chi\right)\left({\eps D_x-\xi\over \eps R}\right)\chi\left({\eps D_x-\xi\over \delta}\right)u^\eps_0\right\|_{L^2(\R^d)}=0.$$
Then, the inequality in formula~(\ref{eq:limu}) becomes an equality.
\end{theorem}

Note that for the sequence initial data~(\ref{eq:data0bis}) introduced previously, one has $u_\xi^0=0$ for $\xi\notin\{\xi_1,\xi_2\}$ and $u_{\xi_j}^0=\theta_j$, $j=1,2$. Identity \eqref{eq:ex} is a consequence of Theorem \ref{theo:disc} in that setting. 

Note also that $u_\xi$ may be identically equal to zero even if the family $(u^\eps_0)_{\eps > 0}$ oscillates in the direction~$\xi$. To see this, simply modulate the waves in example~(\ref{eq:data0}) by an amplitude that concentrates around some point $x_0\in\R^d$:
\begin{equation}\label{eq:cs}
u^\eps_0(x)=\frac{1}{\eps^{d/4}}\theta\left(\frac{x-x_0}{\sqrt{\eps}}\right){\rm e}^{\frac{i}{\eps}\xi_0 \cdot x}
\end{equation}
This corresponds to a coherent state centered at the point $(x_0,\xi_0)$ in phase-space. In this case $u_{\xi_0}=0$ for every $\xi\in\R^d$. Thus, Theorem~\ref{theo:disc} allows us to conclude that the corresponding solutions $\left(u^\eps\right)_{\eps > 0}$ converge to zero in $L^2_{\rm loc}(\R\times\R^d)$ as $\eps \rightarrow 0^+$. 

\medskip

To conclude this section, let us investigate what kind of behavior can be expected when $\lambda$ has degenerate critical points. Suppose $\xi_0\in\Lambda$ and that $\omega_0\in\R^d$ exists such that $\omega_0\in\ker \nabla^2\lambda(\xi_0)$ and $\|\omega_0\|=1$. Let us slightly modify the initial data (\ref{eq:cs}) for $x_0=0$ by introducing a phase shift: 
\begin{equation}\label{data}
u^\eps_0(x)=\frac{1}{\eps^{\alpha d/2}}\theta\left({x\over\eps^\alpha}\right){\rm e}^{{i\over \eps}x\cdot (\xi_0+\eps^\beta\omega_0)},
\end{equation}
where $\theta\in{\mathcal S}(\R^d)$, $\alpha\in [0,1)$ and  $\beta\in (0,1)$ satisfies $\alpha+\beta<1$.
A simple computation shows that these data do not satisfy the assumptions of Theorem~\ref{theo:data}. Again, one has $u_{\xi}^0=0$ for any $\xi\in\R^d$; therefore, if $\xi_0$ were a non-degenerate critical point, Theorem \ref{theo:disc} would imply $\nu_t(dx)=0$, this means that the energy of the modified coherent state \eqref{data} would be dispersed to infinity. 
However, when $\xi_0\in\Lambda$ is degenerate this is no longer the case.
\begin{proposition}\label{prop:disp1}
Assume  $\omega_0\in\ker \nabla^2\lambda(\xi_0)$, $|\omega_0|=1$, $\beta>{2\over 3}$ and $V=0$. Let $(u^\eps)_{\eps_0}$ denote the sequence of solutions to \eqref{eq:disp} issued from the initial data \eqref{data}. 
Then, for every $a<b$ and every $\phi\in{\mathcal C}_0(\R^d)$ the following holds.
\begin{itemize}
\item[(i) ] If $\alpha=0$, then  
$$\lim_{\eps\to 0}\int_a^b\int_{\R^d}\phi(x)\left|u^{\eps}(t,x)\right|^2dxdt=\int_a^b\int_{\R^d}\phi(x) \left|{\rm e}^{-\frac{it}{2} \nabla^2\lambda(\xi_0)D_x\cdot D_x}\theta (x)\right|^2dxdt.$$
\item[(ii) ]  If $\alpha\neq 0$, then 
$$\lim_{\eps\to 0}\int_a^b\int_{\R^d}\phi(x)|u^{\eps}(t,x)|^2dxdt = (b-a) \, \phi(0) \, \|\theta\|_{L^2(\R^d)}^2.$$
\end{itemize}
\end{proposition}
This example also shows that defect measures can be singular when critical points of the symbol are degenerate. In the example above we have:
$$\nu_t(dx)= \|\theta\|_{L^2(\R^d)}^2\delta_0(dx).$$
Of course, this can never occur if $\Lambda$ consists only of non-degenerate critical points, as Theorem \eqref{theo:disc} shows. The proofs of the results in this section are given in Section \ref{sec:two_microlocal}.

\subsection{Non dispersive effects associated to a manifold of critical points}

A natural generalization of the results of the previous section consists in assuming that the set of critical points is a smooth submanifold of $\R^d$. This situation has been examined in detail in \cite{CFM}. Here, in order to keep the presentation reasonably self-contained, we describe the results in the geometrically simpler case in which $\Lambda$ is an affine variety of codimension $0<p\leq d$. After performing a linear change of coordinates in momentum space, we can assume that $\Lambda$ takes the following form.
\begin{enumerate}
\item [\textbf{A2'}] The set $\Lambda$ of critical points of $\lambda$ is of the form: $$\Lambda = \{(\xi=(\xi',\xi'')\in\R^{r}\times\R^p\,:\, \xi''=\xi''_0\},$$
for some $\xi_0''\in\R^p$. Above we have $0<p\leq d$ and $r:=d-p$.
\end{enumerate}
Before stating the main result in this case, we must introduce some notations. We decompose the physical space as $x=(x',x'')\in\R^r\times\R^p$. Given a function $\phi\in L^\infty(\R^d)$, we write $m_\phi(x')$, where $(x')\in\R^{r}$, to denote the operator acting on $L^2(\R^p)$ by multiplication by $\phi(x',\cdot)$:
\begin{equation}\label{eq:op_mult}
m_\phi(x') f(y) = \phi(x',y)f(y), \quad {\rm for} \; f \in L^2(\R^p).
\end{equation}
Note that assumption {\bf A2'} implies that for any $\xi\in\Lambda$ the non-trivial part of the Hessian of $\lambda$ at $\xi$ defines a differential operator $\nabla^2_{\xi''}\lambda(\xi)D_y\cdot D_y$ acting on function defined on $\R^p$ . 

In our next result, the sum over critical points appearing in the statement of Theorem \ref{theo:disc} is replaced by an integral with respect to a certain measure over $\R^r\times\Lambda$, and the Schr\"odinger equation (\ref{eq:schrohprofil}) becomes a Heisenberg equation for a time-dependent family $M$ of trace-class operators acting on $L^2(\R^p)$. More precisely, the operators $M$ depend on $t\in\R$ and on $(x',\xi')\in\R^r\times\R^r$; for every choice of these parameters, $M_t(x',\xi')$ is an element of $\mathcal{L}^1_+\left(L^2(\R^p)\right)$, \emph{i.e.}, it is a positive, Hermitian, trace-class operator acting on  $L^2(\R^p)$. 

\begin{theorem}\label{theo:nondis}
Assume that the sequence of initial data $\left(u^\eps_0\right)_{\eps > 0}$ verifies {\bf A0} and that $\lambda$ and $V$ satisfy {\bf A1} and {\bf A2'}; denote by~$\left(u^\eps\right)_{\eps > 0}$ the corresponding family of solutions to (\ref{eq:disp}). Suppose $(\eps_n)_{n\in\N}$ is a subsequence along which $(|u^{\eps_n}|^2)_{n\in\N}$ converges, in the sence of \eqref{eq:density}, to some defect measure $\nu_t(dx)dt$. Then there exist a positive Radon measure $\nu^0$ defined on $\R^r\times\R^r$ and a measurable family of self-adjoint, positive, trace-class operators 
$$M_0:\R^r\times\R^r \ni (x',\xi')\longmapsto M_0(x',\xi')\in \mathcal{L}_+^1(L^2(\R^r)),\quad {\rm Tr}_{L^2(\R^p)} M_0(x',\xi')=1,$$ 
such that, for almost every $t\in\R$ and every $\phi\in\mathcal{C}_0(\R^d)$ the following holds:
\begin{equation}\label{eq:limubis}
\int_{\R^d}\phi(x)\nu_t(dx)\geq \int_{\R^r\times\R^r}{\rm Tr}_{L^2(\R^p)}\left[m_\phi(x',\xi')M_t(x',\xi')\right]\nu^0(dx',d\xi'),
\end{equation}
and $M \in \mathcal{C}(\R,\mathcal{L}_+^1(L^2(\R^p))$ solves the following Heisenberg equation:
\begin{equation}\label{eq:heis1}
\left\{ \begin{array}{l}
i\partial_t M_t(x',\xi') =\left[\frac{1}{2} \nabla^2_{\xi''}\lambda(\xi',\xi''_0)D_y\cdot D_y + m_{V}(x'), M_t(x',\xi')\right] ,\vspace{0.2cm}\\
M |_{t=0}=M_0.
\end{array} \right.
\end{equation}
Moreover, the measure $\nu^0$ and the family of operators $M_0$ are computed in terms of the sequence initial data $(u^{\eps}_0)_{\eps>0}$. In particular, they do not depend on $\lambda$ or $V$.
\end{theorem}
The nature of the objects involved in this result is described in Section \ref{sec:31}.
As before, a certain non-degeneracy condition on the points of $\Lambda$ implies that the inequality \eqref{eq:limubis} is in fact an identity.

\begin{theorem}\label{theo:simple}
Suppose all the hypotheses of Theorem \ref{theo:nondis} are satisfied. If in addition to those, for every $\xi\in\Lambda$ the rank of the Hessian $\nabla^2 \lambda(\xi)$ is equal to $p$ then \eqref{eq:limubis} is an identity. 
\end{theorem}

When $\Lambda=\{\xi_0\}$ consists of a single critical point, the statements of Theorems \ref{theo:disc} and \ref{theo:nondis} turn out to be completely equivalent. In this case, $r=0$, which forces $\nu^0(dx',d\xi')=\|u_{\xi_0}^0\|^2_{L^2(\R^d)}\delta_0(dx')\delta_0(d\xi')$. In addition, $p=d$, and the operator $M_t$ (which will not depend on $(x',\xi')$) will be the orthogonal projection onto $u_{\xi_0}(t,\cdot)$ in $L^2(\R^d)$. Since $u_{\xi_0}$ solves the Schrödinger equation~(\ref{eq:schrohprofil}), these orthogonal projections satisfy the Heisenberg equation (\ref{eq:heis1}). As it will be clear from the proofs, Theorem \ref{theo:nondis} generalises in a straightforward way to the case that $\Lambda$ is a disjoint union of affine varieties of $\R^d$.

\begin{remark}\label{rem:critical}
As soon as the dimension of $\Lambda$ is strictly positive, the measure $\nu^0$ may be singular with respect to the variable $x'$. This fact implies that the limiting measure of the sequence $\left(|u^{\eps}|^2\right)_{\eps>0}$ may be singular in the variable $x$. Indeed, assume for example $\Lambda=\{\xi''=0\}$, $p\not=0$, and 
$$u^\eps_0(x)=\eps^{\alpha(p-d)\over 2}\theta(x'')\varphi\left({x'-z_0\over \eps^\alpha}\right){\rm e}^{i{x'\cdot \zeta_0\over \eps}},$$
where $\alpha\in (0,1)$, $z_0,\zeta_0\in\R^r$, $\varphi\in{\mathcal C}_0^\infty(\R^{r})$, $\theta\in{\mathcal C}_0^\infty(\R^{p})$ and $\|\theta\|_{L^2(\R^p)}=1$.
Then the measure $\nu^0$ and the operator $M_0$ of Theorem~\ref{theo:nondis} will be: 

\begin{equation}\label{eq:critical}
\nu^0(dx',d\xi')= \| \varphi\|^2_{L^2(\R^{r})}\delta_{z_0}(dx')\delta_{\zeta_0}(d\xi') \qquad {\rm and} \qquad M_0(x',\xi')=|\theta\rangle \langle \theta |,
\end{equation}
see Corolary~\ref{cor:critical} in Section \ref{sec:submanifold}.
\end{remark}

\subsection{Link with smoothing-type estimates}
Since the pioneering works \cite{Ka83,Sjol:87,V88,CS88,KPV91,BAD91} it is well-known that dispersive-type equations develop some kind of smoothing effect. Usually, this is described by means of smoothing-type estimates. Theorems \ref{theo:disc} and \ref{theo:nondis} can be used, in a rather straightforward way, to describe obstructions to the validity of smoothing-type estimates in the presence of non-zero critical points of the symbol $\lambda$. Note that this type of behavior was already described in \cite{H03}; smoothing-type estimates outside the critical points of $\lambda$ were recently presented in \cite{RS16}). We present a simple application of Theorem \ref{theo:disc} to this setting.

\begin{corollary}\label{cor:smoothing}
Suppose {\bf A1}, {\bf A2}, hold and that $\lambda$ has a non-zero critical point $\xi_0$. Then, given any $\delta,s>0$ and any ball $B\subset\R^d$ it is not possible to find a constant $C>0$ such that the estimate
\begin{equation}\label{smoothingestimate}
\int _0^\delta \| |D_x|^s u^\eps(t,\cdot)\|_{L^2(B)}^2 dt \leq C \| u^\eps_0\|_{L^2(\R^d)}^2,
\end{equation}
holds uniformly for every solution $u^\eps$ of  of (\ref{eq:disp}) with initial datum $u^\eps_0\in \mathcal{C}^\infty_0(\R^d)$.
\end{corollary}

\begin{proof}
We argue by contradiction. Suppose the estimate \eqref{smoothingestimate} holds for some $\delta,s,C>0$ and some ball $B$. Let $\theta\in \mathcal{C}^\infty_0(\R^d)$ with $\|\theta\|_{L^2(\R^d)}=1$ and consider the sequence of initial data:
$$u^\eps_0(x):=\theta(x){\rm e}^{i\frac{\xi_0}{\eps}\cdot x}.$$
Clearly $\|u^\eps_0\|_{L^2(\R^d)}=1$ and $(u^\eps_0)$ converges weakly to zero since $\xi_0\neq 0$. Estimate  \eqref{smoothingestimate} then implies that $(u^\eps)$ is bounded in $L^2((0,\delta);H^s(B))$ and Rellich's theorem gives that a subsequence of $(u^\eps)$ converges strongly in $L^2((0,\delta)\times B)$. This limit must be zero, since $(u^\eps)$ weakly converges to zero in that space.

\noindent Now, Theorem \ref{theo:disc} implies that:
$$0\geq |u_{\xi_0}(t,\cdot)|^2dx;$$
in particular, $u_{\xi_0}(t,\cdot)=0$ for every $t\in\R$. But this is a contradiction, since, as $u_{\xi_0}$ is the solution of the Schrödinger equation \eqref{eq:schrohprofil} with initial datum $u^0_{\xi_0}=\theta$, one necessarily has $\|u_{\xi_0}(t,\cdot)\|_{L^2(\R^d)}=1$.
\end{proof}

Of course, Theorem \ref{theo:nondis} gives an analogous consequence when the set of critical points is not isolated.

\medskip

\noindent\textbf{Acknowledgements.} F. Macià has been supported by grants StG-2777778 (U.E.) and MTM2013-41780-P, TRA2013-41096-P (MINECO, Spain). Part of this work was done while V. Chabu was visiting ETSI Navales at Universidad Politécnica de Madrid in the fall of 2015.

\section{The microlocal approach to the problem}\label{sec:wigner}
\subsection{Wigner measures}\label{sec:wigner_distributions}
Wigner distributions provide a useful way for computing weak-$\star$ accumulation points of a sequence of densities $\left(|u^\eps|^2\right)_{\eps>0}$ constructed from a $L^2$-bounded sequence $\left(u^\eps\right)_{\eps>0}$ of solutions to a semiclassical (pseudo) differential equation. They provide a joint position and momentum description of the $L^2$-mass distribution of functions.
The (momentum scaled) Wigner distribution of a function $f\in L^2(\R^d)$ is defined as:
$$W^\eps_f(x,\xi)=\int_{\R^d}f\left(x-\frac{\eps v}{2}\right)\overline{f\left(x+\frac{\eps v}{2}\right)}{\rm e}^{i\xi\cdot v}\frac{dv}{(2\pi)^d}.$$
It enjoys several interesting properties : 
\begin{itemize}
\item $W^\eps_f\in L^2(\R^d\times\R^d)$.
\item Projecting $W_f^\eps$ on $x$ or $\xi$  gives the position or momentum densities of $f$, respectively:
$$\int_{\R^d}W_f^\eps(x,\xi)d\xi=|f(x)|^2,\quad \int_{\R^d}W_f^\eps(x,\xi)dx=\frac{1}{(2\pi\eps)^d}\left|\widehat{f}\left(\frac{\xi}{\eps}\right)\right|^2.$$
Note that in spite of this, $W_f^\eps$ is not positive in general.
\item For every $a\in\mathcal{C}^\infty_0(\R^d\times\R^d)$ one has:
\begin{equation}\label{eq:wbdd}
\int_{\R^d\times\R^d}a(x,\xi)W_f^\eps(x,\xi)dx\,d\xi=(\op_\eps(a)f,f)_{L^2(\R^d)},
\end{equation}
where $\op_\eps(a)$ is the semiclassical pseudodifferential operator of symbol $a$ obtained through the Weyl quantization rule: 
$$\op_\eps(a) f(x)=\int_{\R^{d}\times\R^d} a\left(\frac{x+y}{2},\eps\xi\right) {\rm e}^{i \xi\cdot (x-y)} f(y)dy\,\frac{d\xi}{(2\pi)^d}.$$  
\end{itemize} 
See, for instance, \cite{FollandPhaseSpace} for proofs of these results.

\medskip

If $\left(f^\eps\right)_{\eps>0}$ is a bounded sequence in $L^2(\R^d)$, then $(W^\eps_{f^\eps})_{\eps>0}$ is a bounded sequence of tempered distributions in $\mathcal{S}'(\R^d\times\R^d)$. In addition, every accumulation point of $(W^\eps_{f^\eps})_{\eps>0}$ in $\mathcal{S}'(\R^d\times\R^d)$ is a positive distribution and, therefore, by Schwartz's theorem, an element of $\mathcal{M}_+(\R^d\times\R^d)$, the set of positive Radon measures on $\R^d\times\R^d$. These measures are called \textit{semiclassical} or \textit{Wigner measures}. See references \cite{Ge91,LionsPaul,GerLeich93,GMMP} for different proofs of the results we have presented in here. 

\medskip

Now, if $\mu\in\mathcal{M}_+(\R^d\times\R^d)$ is an accumulation point of $(W^\eps_{f^\eps})_{\eps>0}$ along some sequence $(\eps_k)_{k \in \N}$
and $\left(|f^{\eps_k}|^2\right)_{k \in \N}$ converges weakly-$\star$ towards a measure $\nu$ on $\R^d$, then one has:
$$\int_{\R^d}\mu(\cdot,d\xi)\leq \nu.$$
Equality holds if and only if $\left(f^{\eps_k}\right)_{k \in \N}$ is $\eps$-oscillating in the sense of {\bf A0}
(see \cite{Ge91,GerLeich93,GMMP}). Note also that this implies that $\mu$ is always a finite measure and its total mass is bounded by $\sup_\eps \left\|f^\eps\right\|^2_{L^2(\R^d)}$.

\medskip

This fact justifies the idea of replacing the analysis of energy densities by that of Wigner distributions, which allows one to use a larger set of test functions and to take into account in a more precise way the effects of oscillation of the studied functions, by considering the Fourier variable. 

\medskip 

When the sequence under consideration consists of solutions to the dispersive equation \eqref{eq:disp}, the convergence of the corresponding Wigner distributions towards a Wigner measure still holds provided one averages in time. More precisely, let $(u^{\eps})_{\eps>0}$ be a sequence of solutions to \eqref{eq:disp} issued from a sequence of initial data $(u^\eps_0)_{\eps>0}$ satisfying \textbf{A0}. Then there exist a subsequence $(\eps_k)$ tending to zero as $k\to\infty$ and a $t$-measurable family $\mu_t\in\mathcal{M}_+(\R^d\times\R^d)$ of finite measures, with total mass essentially uniformly bounded in $t\in\R$, such that, for every $\Xi\in L^1(\R)$ and $a\in\mathcal{S}(\R^d\times\R^d)$:
$$\lim_{k\to\infty}\int_{\R\times\R^d\times\R^d} \Xi(t) a(x,\xi)W_{u^{\eps_k}(t)}^{\eps_k}(x,\xi)dx\,d\xi\, dt=\int_{\R\times\R^d\times\R^d}\Xi(t)a(x,\xi)\mu_t(dx,d\xi)dt.$$
Moreover,  for every $\Xi\in L^1(\R)$ and $\phi\in\mathcal{C}_0(\R^d)$: 
$$\lim_{k\to\infty}\int_\R\int_{\R^d} \Xi(t) \phi(x)|u^{\eps_{k}}(t,x)|^2 dx \,dt=\int_\R\int_{\R^d\times\R^d} \Xi(t)\phi(x) \mu_t( dx,d\xi) dt.$$ 
It turns out that the fact that $\left(u^{\eps_k}\right)_{k \in \N}$ is a sequence of solutions to~(\ref{eq:disp}) imposes certain restrictions on the measures $\mu_t$ that can be attained as a limit. In the region of the phase space $\R^d_x\times\R^d_\xi$ where equation (\ref{eq:disp}) is dispersive (\textit{i.e.}, away from the non-zero critical points of $\lambda$), the energy of the sequence $\left(u^{\eps_k}\right)_{k \in \N}$ is dispersed at infinite speed towards infinity. More precisely, Wigner measures $\mu_t$ satisfy:
\begin{equation}\label{prop:loc}
{\rm supp}\, \mu_t\subset  \R^d\times\Lambda. 
\end{equation}
Proofs of these results can be found in \cite{CFM}. 

\medskip

In what follows, we investigate the precise structure of Wigner measures $\mu_t$. In order to get a better description of $\mu_t$ on $\R^d\times\Lambda$ we shall perform a second microlocalisation of the solutions above $\R^d\times\Lambda$. 

\subsection{Two-microlocal Wigner measures}\label{sec:31}
Two-microlocal Wigner measures are objects designed to describe in a precise way oscillation and concentration effects exhibited by sequences of functions on a submanifold $X\subset\R^d\times\R^d$ of phase space. Roughly speaking, the idea consists in working in an enlarged phase space by adding an additional variable that will give a more precise description of the behavior of the Wigner functions close to the set $X$. These measures were introduced in \cite{Fthese:95,F2micro:00, MillerThesis, Nierscat} and further developed in \cite{FG:02} in a slightly different framework.
 
Here, we are particularly interested in the situation where $X=\Lambda$, the set of critical points of the symbol $\lambda$. In any case, the theory can be developed without assuming that we are dealing with solutions to an evolution equation, and the submanifold $X$ is not required to have some dynamical meaning. It is convenient to present first the results in this more general framework.

\medskip

We are first going to assume that $X$ is an affine manifold of $\R^d$ with codimension $p$ given by the equations:
$$ \xi_{r+1}=\xi_0^{1},\; ... \;,\xi_d=\xi_0^p,\quad {\rm for } \quad\xi_0''=(\xi_0^{1},...\,,\xi_0^p)\in\R^p,\quad r:=d-p,$$
and, given $\xi\in\R^d$, we will set $\xi=(\xi',\xi'')$ with $\xi''=(\xi_{r+1},...\, ,\xi_d)$.

\begin{remark} 
Note that any submanifold of codimension $p$ in $\R^d$ can be locally identified to a linear space $\{\xi''=0\}$ by using a suitable coordinate system, which may be used to extend the analysis of this section to this more general setting. However, in doing so, it turns out that the dependence on the choice of local coordinates becomes an issue and requires special care. We refer the reader to \cite{CFM} for precise results in that setting. 
\end{remark}

\medskip

Now we will extend the phase space $\R^d_x\times\R^d_\xi$ with a new variable $\eta\in \overline{ \R^p}$, where $\overline{\R^p}$ is the compactification of $\R^p$ obtained by adding a sphere ${\bf S}^{p-1}$ at infinity. The test functions associated to this extended phase space are functions $a\in{\mathcal C}^\infty(\R^d_x\times\R^d_\xi\times\R^p_\eta)$ which satisfy the two following properties:
\begin{enumerate}
\item there exists a compact $K \subset \R^{2d}$ such that, for all $\eta\in\R^p$, the map $(x,\xi)\longmapsto a(x,\xi,\eta)$ is a smooth function compactly supported in $K$;
\item there exists a function $a_\infty$ defined on $\R^d\times\R^d\times{\bf S}^{p-1}$ and $R_0>0$ such that, 
$${\rm  if}\;\; \|\eta\|>R_0, \;\;{\rm then}\;\; a(x,\xi,\eta)=a_\infty\left(x,\xi,\frac{\eta}{\|\eta\|}\right).$$  
\end{enumerate}
We denote by $\mathcal{S}^0(p)$ the set of such functions; to every $a\in \mathcal{S}^0(p)$ we associate a pseudodifferential operator $\op_\eps^\sharp(a)$ as follows:  
\begin{equation}\label{def:2quant}
\op_\eps^\sharp(a)=\op_\eps(a^\sharp_\eps), \quad {\rm where } \quad a^\sharp_\eps(x,\xi)=a\left(x,\xi,\frac{\xi''-\xi_0''}{\eps }\right).
\end{equation}
In the above formula, the additional variable $\eta = \frac{\xi''-\xi_0''}{\eps}$ is introduced to capture in greater detail the concentration properties of a sequence of functions onto the set $\{\xi''=\xi_0''\}$. Moreover, notice that 
\begin{equation}\label{beth}
\op_\eps^\sharp(a)={\rm e}^{-i{x''\cdot \xi''_0\over\eps}}\op_1(a(x,\eps\xi',\xi_0''+\eps\xi'',\xi'')) {\rm e}^{i{x''\cdot \xi''_0\over\eps}},
\end{equation}
which implies in particular that the family $(\op_\eps^\sharp(a))_{\eps > 0}$ is uniformly bounded in ${\mathcal L}\left(L^2(\R^d)\right)$.

\medskip 

Now, let $\left(u^\eps\right)_{\eps > 0}$ be a sequence in ${\mathcal C}^0\left(\R,L^2(\R^d)\right)$ (so each $u^\eps$ is a continuous function of time into $L^2(\R^d)$) satisfying the uniform bounds: 
$$\exists C_0>0,\;\; \forall t\in \R, \;\;\| u^\eps(t,\cdot\,)\|_{L^2(\R^d)} \leq C_0.$$
Note that this is the case if $\left(u^\eps\right)_{\eps > 0}$ is a family of solutions to \ref{eq:disp} evolved from a sequence of initial data bounded in $L^2(\R^d)$. We will use these functions to define a linear functional $I^\eps_{u^\eps}$ acting on $\mathcal{S}^0(p) \times L^1(\R)$ as:
$$I^\eps_{u^\eps}(a,\Xi) = \int_\R\Xi(t) \left(\op_\eps^\sharp(a)u^\eps(t,\cdot),u^\eps(t,\cdot)\right)_{L^2(\R^d)}dt.$$

These functionals are actually lifts to the extended phase space $\R^d_x\times\R^d_\xi\times\overline{ \R^p}_\eta$ of the Wigner distributions $W^\eps_{u^\eps(t,\cdot\,)}$. To see this, note that any function $a\in\mathcal{C}_0^\infty(\R^d_x\times \R^d_\xi)$ can be identified to an elements of $\mathcal{S}^0(p)$ that is constant in the variable $\eta$; clearly, under this identification one has: 
\begin{equation}\label{eq:eqquant}
\op_\eps^\sharp(a)=\op_\eps(a),
\end{equation}
which implies, for these kind of $a$ independent of $\eta$:
$$I^\eps_{u^\eps}(a,\Xi) =\int_\R\int_{\R^d\times\R^d} \Xi(t)a(x,\xi)W^\eps_{u^\eps(t,\cdot\,)}(x,\xi)\,dx\,d\xi, \quad \forall a\in\mathcal{C}^\infty_0(\R^d\times\R^d).$$
Therefore, letting $\mu_t$ denote the Wigner measures of $\left(u^\eps(t,\cdot)\right)_{\eps > 0}$ as described in Section \ref{sec:wigner_distributions}, we have by dominated convergence and the definition of $\mu_t$:
\begin{equation}\label{eq:lift}
I^\eps_{u^\eps}(a,\Xi) \Tend{\eps}{0} 
\int_\R\Xi(t)\int_{\R^d\times\R^d} a\left(x,\xi\right) \mu_t(dx,d\xi) dt, \quad \forall a\in\mathcal{C}^\infty_0(\R^d\times\R^d).
\end{equation}

Nevertheless, when the convergence of $\left(I^\eps_{u^\eps}\right)_{\eps > 0}$ is tested against general functions on the extended phase, the resulting accumulation points have some additional structure:
\begin{proposition}\label{prop:2micro}
Suppose that $\left(u^\eps\right)_{\eps > 0}$ and $\mu_t$ are as above. Then, up to the extraction of a sequence $(\eps_k)_{k \in \N}$, there exist a $L^\infty$-map $\gamma:t\longmapsto \gamma_t$ taking values in the set of positive Radon measures on $\R^{d}\times\R^r\times{\bf S}^{p-1}$ and a $L^\infty$-map $\mathtt{M}:t\longmapsto \mathtt{M}_t$ into the set of operator-valued positive measures on $\R\times\R^{2r}$ that are trace class operators on $L^2(\R^p_y)$ such that, for every $a\in \mathcal{S}^0(p)$ and $\Xi\in L^1(\R)$, 
\begin{eqnarray}\nonumber
I^\eps_{u^\eps}(a,\Xi) &\Tend{\eps}{0} &
\int_\R\Xi(t)\int_{\{\xi''\not=\xi_0''\}} a_\infty\left(x,\xi,\frac{\xi''-\xi_0''}{\|\xi''-\xi_0''\|}\right) \mu_t(dx,d\xi) dt \\
\nonumber & & 
+\int_\R\Xi(t) \int_{\R^d\times\R^{r}\times {\bf S}^{p-1}} a_\infty(x,\xi',\xi_0'',\omega) \gamma_t(dx,d\xi',d\omega) dt\\
\label{2mic:coord} & & 
+ \int_\R \Xi(t) \int_{\R^{2r}} {\rm Tr}_{L^2(\R^p)}  \left[ a^W(x',y,\xi',\xi_0'',D_y) \mathtt{M}_t(dx',d\xi') \right] dt,
\end{eqnarray}
where, for every $(x',\xi')\in \R^{2r}$, $a^W(x',y,\xi',\xi_0'',D_y)$ denotes the pseudodifferential operator acting on $L^2(\R^p)$ obtained by the Weyl quantization of the symbol $(y,\eta)\longmapsto a(x',y,\xi',\xi_0'',\eta)$.  
\end{proposition}

The proof of this result is essentially identical to that of Theorem 1 in~\cite{F2micro:00} (except for the fact that here everything depends on $t$); 
see also~\cite{Fthese:95,AM:14,AFM:15} for very closely related results in a slightly different context. 

\medskip

In order to enlighten the nature of the different objects involved in formula \eqref{2mic:coord}, we emphasize the following characterization : let 
$\chi\in{\mathcal C} _0^\infty (\R^p)$ be a cut-off function such that $0\leq \chi\leq 1$, $\chi(\eta)=1$ for $\|\eta\|\leq 1$ and  $\chi(\eta)=0$ for $\|\eta\|\geq 2$; then, one has: (see again the proof of Theorem 1 in \cite{F2micro:00}, or \cite{Fthese:95,AM:14,AFM:15}):
\begin{itemize}
\item[(i)] the measure $\gamma_t$ in Proposition \ref{prop:2micro} is obtained through the limiting procedure
$$
\int_\R\Xi(t) \int_{\R^d\times\R^{r}\times {\bf S}^{p-1}} a_\infty(x,\xi',\xi_0'',\omega) \gamma_t(dx,d\xi',d\omega) dt
= \lim_{\delta\to 0}\lim_{R\to\infty} \lim_{\eps\to 0} I^\eps_{u^\eps}(a^{R,\delta},\Xi),
$$
where
\begin{equation}\label{eq:testinfty}
a^{R,\delta}(x,\xi,\eta)=a(x,\xi,\eta)\chi\left(\frac{\xi''-\xi_0''}{\delta}\right) \left(1-\chi\left(\frac{\eta}{R}\right)\right);
\end{equation} 

\item[(ii) ] the measure $\mathtt{M}_t$ in Proposition \ref{prop:2micro} is obtained as the iterated limits
$$
\int_\R \Xi(t) \int_{\R^{2r}} {\rm Tr}_{L^2(\R^p)} [ a^W(x',y,\xi', \xi_0'',D_y) \mathtt{M}_t(dx',d\xi') ] dt
=\lim_{\delta\to 0}\lim_{R\to\infty} \lim_{\eps\to 0} I^\eps_{u^\eps}(a_{R,\delta},\Xi),
$$
where
\begin{equation}\label{eq:testfin}
a_{R,\delta}(x,\xi,\eta)=a(x,\xi,\eta)\chi\left(\frac{\xi''-\xi_0''}{\delta}\right) \chi\left(\frac{\eta}{R}\right).
\end{equation}
\end{itemize}

The presence of the cut-off $\chi(\eta/R)$ explains the different roles played by $\gamma_t$ and $\mathtt{M}_t$. The measure~$\mathtt{M}_t$ captures the fraction of the $L^2$-mass of the sequence $(u^\eps)$ that concentrates onto $\{\xi''=\xi_0''\}$ at rate precisely $\eps$. The measure $\gamma_t$, on the other hand, describes how the $L^2$-mass of the sequences concentrates on $\{\xi''=\xi_0''\}$ at a slower rate.

\medskip

Besides, it is convenient to use a decomposition of $\mathtt{M}_t$ based on the Radon-Nikodym Theorem.
Define the map $\nu : t\longmapsto \nu_t$  by 
$$\nu_t(dx',d\xi')={\rm Tr}_{L^2(\R^{p})} \mathtt{M}_t(dx',d\xi').$$
This $L^\infty$-map
is valued in the set of positive measures on $\R^{2r}$ and there exists a measurable map $M : (t, x',\xi')\longmapsto M_t(x',\xi')$ valued in the set of self-adjoint, positive, trace-class operators on $L^2(\R^p)$ such that 
$$\mathtt{M}_t(dx',d\xi')= M_t(x',\xi')  \nu_t(dx',d\xi').$$
Note that, by construction, we have 
${\rm Tr}_{L^2(\R^{p})} M_t(x',\xi')=1.$
We are then left with three objects, $\gamma_t$, $\nu_t$ and~$M_t$. 

\medskip

Note that the results we have presented so far hold without assuming that $u^\eps$ solves an evolution equation, nor that $\{\xi''=\xi_0''\}$ is the set of critical points of the function $\lambda$. The fact that the sequence $\left(u^{\eps_k}(t,\cdot\,)\right)_{k \in \N}$ generating $\mu_t$ consists of solutions to equation (\ref{eq:disp}) and that $\Lambda$ is the set of critical points of its symbol implies additional regularity and propagation properties on the measures $\gamma_t$ and $\mathtt{M}_t$ that we will use in the next section. Let us anticipate that the latter propagates following a Heisenberg equation, whereas $\gamma_t$, enjoys an additional geometric invariance. Finally, it is not hard to prove (though we will not do that here, see \cite{CFM}) that $\nu_t$ does not depend on $t$, and in fact $\nu_t=\nu^0$, which is the measure appearing in the statement of Theorem \ref{theo:nondis}, only depends on the sequence of initial data.

\medskip

Let us mention that the use of two-microlocal semiclassical measures for dispersive equations was initiated in \cite{MaciaTorus}, in the context of the Schrödinger equation on the torus. These results were largely extended and improved in subsequent works \cite{AM:14,AFM:15}. The reader might find interesting to compare the results of the present note to those in the aforementioned references.


\section{The particular case: countable critical points}\label{sec:two_microlocal}

This section is mainly devoted to a sketch of the proof of Theorems~\ref{theo:disc}, \ref{theo:data}, and to the analysis of the examples of Proposition~\ref{prop:disp1}.

\subsection{Two microlocal Wigner measures associated to a critical point}\label{sec:isolated_points}

Our goal in this section will be to compute the restriction to $\{\xi=\xi_0\}$, with $\xi_0\in\Lambda$, of the semiclassical measure $\mu_t$ associated to sequences of solutions to (\ref{eq:disp}) in terms of quantities that depend only on the sequence of initial data. 

\medskip

The results of the previous section applied to the particular case $\{\xi=\xi_0\}$, $p=d$, ensure the existence of measures $\gamma_t\in\mathcal{M}_+(\R^d\times{\bf S}^{d-1})$ and of a positive family of Hermitian operators $M_t\in\mathcal{L}^1(L^2(\R^d))$ such that, for all $(a,\Xi)\in\mathcal{S}^0(d)\times L^1(\R)$,
\begin{equation}\label{eq:sgamma0}
\int_\R\Xi(t) \int_{\R^d\times{\bf S}^{d-1}} a_\infty(x,\xi_0,\omega) \gamma_t(dx,d\omega) dt
 =\lim_{\delta\to 0}\lim_{R\to\infty} \lim_{\eps\to 0} I^\eps_{u^\eps}(a^{R,\delta},\Xi)
\end{equation}
and
\begin{equation}\label{eq:sM0}
 \int_\R \Xi(t) {\rm Tr}_{L^2(\R^d)} [ a^W(y,\xi_0,D_y)   {M}_t] dt
=\lim_{\delta\to 0}\lim_{R\to\infty} \lim_{\eps\to 0} I^\eps_{u^\eps}(a_{R,\delta},\Xi),
\end{equation}
where
$a^{R,\delta}(x,\xi,\eta)$ and $a_{R,\delta}(x,\xi,\eta)$ are defined in~(\ref{eq:testinfty}) and (\ref{eq:testfin}) respectively. 
The localization property~\eqref{prop:loc} and Proposition~\ref{prop:2micro}, together with identity~(\ref{eq:lift}), then assert that, for every $b\in\mathcal{C}^\infty_0(\R^d\times\R^d)$ and a.e. $t\in\R$:
\begin{equation}\label{eq:mopsimple}
\int_{\{\xi=\xi_0\}}b(x,\xi)\mu_t(dx,d\xi) = \int_{\R^d\times{\bf S}^{d-1}} b(x,\xi_0)\gamma_t(dx,d\omega) +{\rm Tr}_{L^2(\R^d)}[ b^W(\,\cdot\,,\xi_0) M_t].
\end{equation}
The fact that $u^\eps$ solves equation (\ref{eq:disp}) implies that $\gamma_t$ and $M_t$ enjoy the following additional properties:
\begin{theorem}\label{theo:2micro}
Let $(u^{\eps})_{\eps > 0}$ be a sequence of solutions to (\ref{eq:disp}) issued from a $L^2(\R^d)$-bounded sequence of initial data $(u^\eps_0)_{\eps > 0}$, then:
\begin{itemize}
\item[(i)] For almost every $t \in \R$, the measure $\gamma_t$ is invariant through the flow 
$$\phi^2_s : \R^d\times{\bf S}^{d-1}\ni (x,\omega) \longmapsto (x+s \, \nabla^2\lambda(\xi_0)\omega,\omega)\in\R^d\times{\bf S}^{d-1}.$$
\item[(ii)] $\,M_t=|u_{\xi_0}(t,\cdot)\rangle \langle u_{\xi_0}(t,\cdot)|$, where $u_{\xi_0}$ solves 
\begin{equation}\label{eq:Phi}
\left\lbrace 
\begin{array}{l}
i\partial_t u_{\xi_0}(t,x)= \frac{1}{2}\nabla^2\lambda(\xi_0)D_x\cdot D_x  u_{\xi_0}(t,x) +V(x)u_{\xi_0}(t,x), \vspace{0.2cm} \\
u_{\xi_0}|_{t=0}(x)=u_{\xi_0}^0(x),
\end{array}
\right.
\end{equation}
and $u_{\xi_0}^0$ is a weak limit in $L^2(\R^d)$ of $({\rm e}^{-\frac{i}{\eps} \xi_0\cdot x} u^\eps_0)_{\eps > 0}$ when $\eps \to 0^+$.
\end{itemize}
\end{theorem}

The localization property \eqref{prop:loc} and Corollary~\ref{cor:mu} below imply together Theorem~\ref{theo:disc}.

\begin{corollary}\label{cor:mu}
For every $\xi_0\in\Lambda$ and almost every $t \in \R$ one has 
$$\mu_t (dx,d\xi)\rceil_{\{\xi=\xi_0\}} \geq  |u_{\xi_0}(t,x)|^2 dx \,\delta_{\xi_0}(d\xi),$$
with equality if $\xi_0$ is a non-degenerate critical point.
\end{corollary}
\begin{proof}
We are going to show that the  measure $\gamma_t$ vanishes identically if $\xi_0$ is non-degenerate. This is a consequence of the following result, whose proof can be found in \cite{CFM}.
\begin{lemma}\label{lem:classicdisp} Let be $\Phi_s:\R^d\times\R^d\To\R^d\times\R^d$ a flow satisfying: for every compact $K\subset\R^d\times\R^d$ containing no stationary points of $\Phi$, there exist constants $\alpha,\beta>0$ such that:
$$\alpha |s| - \beta \leq \|\Phi_s(x,\xi)\| \leq \alpha|s|+\beta,\quad\forall(x,\xi)\in K\;\forall s \in \R.$$ 
Moreover, let $\mu$ be a finite, positive Radon measure on $\R^d\times\Omega$ that is invariant by the flow $\Phi_s$. Then $\mu$ is supported on the set of stationary points of  $\Phi_s$.
\end{lemma}
When this lemma is applied to the measure $\gamma_t$ and the flow $\phi^2_s$, one finds out that $\gamma_t=0$. 
\end{proof}

\begin{remark}\label{rem:memesmesures}
The formula in Corollary \ref{cor:mu} shows in particular that the semiclassical measure $\mu_t$ is not uniquely determined by the semiclassical measure $\mu_0$ of the sequence of initial data. Suppose that $\xi_0$ is a non-degenerate critical point; if $u^\eps_0=\theta(x){\rm e}^{{i\over \eps}\xi_0\cdot x}$, $\|\theta\|_{L^2(\R^d)}=1$, then $u_{\xi_0}^0=\theta\neq 0$ and $\mu_0 = dx\otimes\delta_{\xi_0}$ and Corollary \ref{cor:mu} tells us that:
$$\mu_t(dx,d\xi)=|u_{\xi_0}(t,x)|^2dx\,\delta_{\xi_0}(d\xi)\neq 0.$$
However, if we choose initial data $v^\eps_0(x)=\theta(x){\rm e}^{{i\over \eps}(\xi_0+\eps^\beta\xi_0)\cdot x}$ with $\beta\in (0,1)$, they have the same semiclassical measure $\mu_0$ as $(u^\eps_0)_{\eps > 0}$, whereas  Corollary \ref{cor:mu} now shows that the  measure $\mu_t(dx,d\xi)$  is $0$, since any weak limit of ${\rm e}^{-{i\over \eps}\xi_0\cdot x}v^\eps_0$ is $0$.
\end{remark}

\begin{proof}[Proof of Theorem \ref{theo:2micro}.]
Let us start proving part (i), namely the invariance of $\gamma_{t}$. Let $a\in{\mathcal S}_0(d)$ and  $a^{R,\delta}$ as defined in~(\ref{eq:testinfty}).
We set  
$$\widetilde{\phi}_{s}^2(x,\xi,\eta):= (x+s\nabla^2 \lambda (\xi_0) {\eta\over|\eta|}, \xi, \eta),\quad (x,\xi,\eta)\in\R^d\times\R^d\times\R^d$$
and we note that  $a^{R,\delta}\circ \widetilde \phi_s^2$ also is a symbol of ${\mathcal S}^0(d)$ (in particular, it is smooth because it is supported on $|\eta|>R$). Besides 
$$\left(a^{R,\delta}\circ \widetilde \phi_s^2\right)_\infty= a_\infty\circ \phi_s^2.$$
Our aim is to prove that for all $\Xi\in{\mathcal C}_0^\infty(\R_t)$, as $\eps$ goes to $0$, then $R$ to~$+\infty$ and finally $\delta$ to $0$,
\begin{equation}\label{claim80}
I^\eps_{u^\eps} \left(a^{R,\delta}\circ \widetilde \phi_s^2,\Xi\right)= I^\eps_{u^\eps} \left(a^{R,\delta},\Xi\right)+o(1).
\end{equation}
We observe that the quantification of $a^{R,\delta}\circ \widetilde \phi_s^2$ has the following property
\begin{eqnarray*}
\left(a^{R,\delta}\circ \widetilde \phi_s^2\right)_\eps^\sharp &=& a^{R,\delta}\left(x+{s\over |\xi-\xi_0| }\nabla^2\lambda(\xi_0)(\xi-\xi_0),\xi,{\xi-\xi_0\over\eps}\right)\\
&=& a^{R,\delta}\left(x+{s\over |\xi-\xi_0|}
\nabla\lambda(\xi),\xi,{\xi-\xi_0\over\eps}\right) + r_{\eps,\delta,R}(x,\xi)\end{eqnarray*}
with 
$$\| \op_\eps(r_{\eps,\delta,R})\|_{{\mathcal L}(L^2(\R^d))}=O(\delta).$$
For this, we have used that $|\xi-\xi_0|<\delta$ on the support of $a^{R,\delta}$, that $\nabla \lambda (\xi_0)=0$ and that there exists a smooth bounded tensor of degree~$3$, $\Gamma$, with bounded derivatives such that 
$$\nabla\lambda(\xi)=\nabla^2\lambda(\xi_0)(\xi-\xi_0) +\Gamma(\xi) [\xi-\xi_0,\xi-\xi_0].$$
As a consequence, the claim~(\ref{claim80}) is equivalent to proving that for all $\Xi\in{\mathcal C}_0^\infty(\R_t)$, as $\eps$ goes to $0$, then $\delta$ to $0$ and $R$ to~$+\infty$,
$$
I^\eps_{u^\eps} \left(b^{R,\delta}_\eps(s),\Xi\right) =  I^\eps_{u^\eps} \left(a^{R,\delta},\Xi\right)+o(1)
$$
where $b^{R,\delta}_\eps(s)$ is the symbol 
$$b^{R,\delta}_\eps (s)(x,\xi)= a^{R,\delta}\left(x+{s\over |\xi-\xi_0| }\nabla^2\lambda(\xi_0)(\xi-\xi_0),\xi,{\xi-\xi_0\over\eps}\right).$$
To this aim, we will show that for all $\Xi\in{\mathcal C}_0^\infty(\R_t)$, as $\eps$ goes to $0$, then $R$ to~$+\infty$ and finally $\delta$ to $0$,
\begin{equation}\label{claim81}
I^\eps_{u^\eps} \left(\partial_s b^{R,\delta}_\eps(s),\Xi\right) = o(1).
\end{equation}
We observe that 
\begin{eqnarray*}
\partial_s b^{R,\delta}_\eps(s)(x,\xi)& = & |\xi-\xi_0|^{-1} \nabla\lambda(\xi)  \cdot\nabla_x a^{R,\delta}\left(x+{s\over |\xi-\xi_0| }\nabla^2\lambda(\xi_0)(\xi-\xi_0),\xi,{\xi-\xi_0\over\eps}\right)\\
& = & \nabla\lambda(\xi) \cdot \nabla_x  \widetilde b^{R,\delta}_\eps (s,x,\xi),
\end{eqnarray*}
with 
$$\widetilde b^{R,\delta}_\eps (s,x,\xi):=  |\xi-\xi_0|^{-1} b_\eps^{R,\delta}(s)(x,\xi).$$ 
This function satisfies:  for all $\alpha\in\N^d$ there exists a constant $C_\alpha$ such that  for all $x,\xi\in\R^d$
$$|\widetilde b_\eps^{R,\delta} (s,x,\xi)| +|\partial_x^\alpha \widetilde b_\eps^{R,\delta}(x,\xi)|\leq C_\alpha(R\eps)^{-1}.$$
By the symbolic calculus of semiclassical pseudodifferential operators, we have 
\begin{eqnarray*}
\op_\eps^\sharp\left( \nabla\lambda(\xi)\cdot \nabla_x \widetilde b_\eps^{R,\delta}  \right)
&=&\frac{i}{\eps}\left[ \lambda(\eps D),\op_\eps \left(\widetilde b_\eps^{R,\delta} \right) \right] +O(\eps)+O(1/R)\\
&=& \frac{i}{\eps}\left[ \lambda(\eps D_x)+\eps^2V,\op_\eps\left(\widetilde b_\eps^{R,\delta} \right) \right] +O(1/R)+O(\eps).
\end{eqnarray*}
On the other hand,
$${d\over dt} \left( \op_\eps\left(\widetilde b_\eps^{R,\delta} \right) u^\eps(t),u^\eps(t)\right) = {i\over \eps^2} \left( \left[ \lambda(\eps D_x)+\eps^2V,\op_\eps\left(\widetilde b_\eps^{R,\delta}\right) \right]  u^\eps(t),u^\eps(t)\right)+O(\eps).$$
Therefore 
\begin{eqnarray*}
I^\eps_{u^\eps} \left(\partial_s b^{R,\delta}_\eps(s),\Xi\right) &=& \int \Xi(t)\left( \op_\eps^\sharp\left( \nabla\lambda(\xi)\cdot \nabla_x \widetilde b_\eps^{R,\delta}  \right)  u^\eps(t),u^\eps(t)\right) dt \\
& = & {i\over \eps}  \int \Xi(t)\left( \left[ \lambda(\eps D),\op_\eps \left(\widetilde b_\eps^{R,\delta} \right) \right]  u^\eps(t),u^\eps(t)\right) dt +O(1/R)+O(\eps) \\
& = & \eps  \int \Xi(t){d\over dt} \left( \op_\eps\left(\widetilde b_\eps^{R,\delta} \right) u^\eps(t),u^\eps(t)\right)  dt +O(1/R)+O(\eps) \\
& = & -\eps  \int \Xi'(t) \left( \op_\eps\left(\widetilde b_\eps^{R,\delta} \right) u^\eps(t),u^\eps(t)\right)  dt +O(1/R)+O(\eps)\\
& = & O(1/R) +O(\eps),
\end{eqnarray*}
which gives~(\ref{claim81}), thus~(\ref{claim80}), and concludes the proof. 

\medskip 

To prove part (ii) one starts noticing that, by symbolic calculus and~(\ref{beth}),
$$
\left(\op_\eps^\sharp\left(a_{R,\delta} \right)u^\eps(t),u^\eps(t)\right)=\left(\op_1(A^\eps_{R,\delta})\Phi^\eps(t),\Phi^\eps(t)\right),
$$
where 
$$\Phi^\eps(t,x):={\rm e}^{-\frac{i}{\eps}\xi_0\cdot x}u^\eps(t,x),\quad A^\eps_{R,\delta}(x,\xi):=a_{R,\delta}\left(x,\xi_0+\eps\xi,\xi\right).$$ 
Since $u^\eps$ solves \eqref{eq:disp}, one sees that $\Phi^\eps$ satisfies 
$$i\partial_t \Phi^\eps (t,x)=\frac{1}{\eps^2}\lambda(\xi_0+\eps D_x)\Phi^\eps (t,x)+ V(x)\Phi^\eps(t,x)+ O(\eps).$$
A Taylor expansion for $\lambda(\xi)$ around $\xi_0$ shows that setting 
$$u^\eps_{\xi_0}(t,x): = {\rm e}^{\frac{it}{\eps^2}\lambda(\xi_0)} \Phi^\eps(t,x),$$
then $u^\eps_{\xi_0}$ solves in $L^2(\R^d)$, 
$$i\partial_t u^\eps_{\xi_0}(t,x) = \nabla^2 \lambda(\xi_0)D_x\cdot D_x \, u^\eps_{\xi_0} (t,x)+ V(x)u^\eps_{\xi_0}(t,x) +O(\eps).$$
We still have, 
$$\left(\op_\eps^\sharp\left(a_{R,\delta} \right)u^\eps(t),u^\eps(t)\right)
=\left( \op_1(A^\eps_{R,\delta})u^\eps_{\xi_0}(t),u^\eps_{\xi_0}(t)\right).$$
On the other hand, 
\begin{eqnarray*}
A^\eps_{R,\delta}(x,\xi)&=&a\left(x,\xi_0+\eps\xi,\xi\right) \chi\left({\eps \xi/ \delta}\right)\chi\left({\xi/ R}\right)\\
&=&a\left(x,\xi_0,\xi\right)\chi\left({\xi/ R}\right)+O(\eps) =A^0_{R}(x,\xi)+O(\eps).
\end{eqnarray*}
Notice that the remainder depends on $R$ and $\delta$, but that this is harmless since we shall first let $\eps$ go to $0$. 
Using again the symbolic calculus, we write 
$$\op_1(A^\eps_{R,\delta})=\op_1(A^0_{R})+O(\eps).$$
Therefore,
\begin{equation}\label{eq:conve}
\lim_{\eps\to 0^+}I^\eps_{u^\eps}(a_{R,\delta},\Xi)=\lim_{\eps\to 0^+}\int_\R\Xi(t)\left( \op_1(A^0_{R})u^\eps_{\xi_0}(t),u^\eps_{\xi_0}(t)\right)dt.
\end{equation} 
By \cite[Lemma 4.26]{Zwobook}, $\op_1(A^0_{R})$ is a compact operator on $L^2(\R^d)$; therefore, if $$u^{\eps_k}_{\xi_0}(0,\cdot)\rightharpoonup u^0_{\xi_0}$$ along some subsequence $(\eps_k)_{k\in\N}$, it follows that, for every $t\in\R$,
$$\lim_{k\to\infty}\left( \op_1(A^0_{R})u^{\eps_k}_{\xi_0}(t),u^{\eps_k}_{\xi_0}(t)\right)=\left( \op_1(A^0_{R})u_{\xi_0}(t),u_{\xi_0}(t)\right),
$$
where $u_{\xi_0}$ solves:
$$i\partial_t u_{\xi_0}(t,x)= \nabla^2 \lambda(\xi_0)D_x\cdot D_x \, u_{\xi_0} (t,x)+ V(x)u_{\xi_0}(t,x),\quad u_{\xi_0}|_{t=0}(x)=u^0_{\xi_0}(x).$$
In particular, if the convergence \eqref{eq:conve} takes place then the sequence $(u^{\eps}_{\xi_0}(0,\cdot))$ must have a unique weak accumulation point and:
$$\lim_{\eps\to 0^+}I^\eps_{u^\eps}(a_{R,\delta},\Xi)=\int_\R\Xi(t){\rm Tr}_{L^2(\R^d)}[\op_1(A^0_{R})| u_{\xi_0}(t)\rangle\langle u_{\xi_0}(t)|]dt.
$$
The result follows from \eqref{eq:sM0} by letting $R$ go to $+\infty$.   
\end{proof}

In order to prove Theorem~\ref{theo:data}, notice that the assumption that is made in its statement implies that $\gamma_0=0$  (by the caracterization of $\gamma$ in (\ref{eq:sgamma0})), and the result comes from the conservation of mass of $\gamma_t$:
\begin{lemma}
For all $t\in\R$, 
$$\int_{\R^d\times{\bf S}^{d-1}} \gamma_t(dx,d\omega)= \int_{\R^d\times{\bf S}^{d-1}} \gamma_0(dx,d\omega).$$
\end{lemma}

\begin{proof}
Using the characterization in (\ref{eq:sgamma0}), we have, for $R,\delta>0$ and $\chi$ as in~(\ref{eq:testinfty}):
$$J^\eps_{R,\delta}(t)=\left(  (1-\chi)\left({\eps D_x-\xi_0\over R\eps}\right)\chi\left({\eps D_x-\xi_0\over \delta}\right)   u^\eps(t,\cdot\,), u^\eps(t,\cdot\,)\right).$$
Using the dynamical equation \eqref{eq:disp}, we obtain 
\begin{eqnarray*}
{d\over dt} J^\eps_{R,\delta}(t) &=&  -i \left(\left[   (1-\chi)\left({\eps D_x-\xi_0\over R\eps}\right)\chi\left({\eps D_x-\xi_0\over \delta}\right)  ,V\right]   u^\eps(t,\cdot),u^\eps(t,\cdot)\right)\\
& =& O(\eps)+O(1/R)+O(\delta)
\end{eqnarray*}
by semiclassical symbolic calculus. Therefore, taking limits in all the parameters one concludes. 
\end{proof}

\subsection{Degenerate critical points}\label{sec:deg}
In this section we focus on the situation of Proposition~\ref{prop:disp1} with 
 the family of initial data given by~(\ref{data}). These concentrate microlocally onto $\xi_0$, and the two-microlocal measures associated with them depend on the value of $\alpha$:
\begin{itemize}
\item[(i)] If $\alpha=0$, then $\gamma_0(dx,d\omega)=|\theta(x)|^2dx\,\delta_{\omega_0}(d\omega)$ and $M_0=0$.\smallskip
\item[(ii)] If $\alpha\not=0$, then $\gamma_0(dx,d\omega)=\delta_0(dx)\, \delta_{\omega_0}(d\omega)$ and $M_0=0$.
\end{itemize}

Comparatively, for the data in (\ref{eq:data0}) (which corresponds to $\alpha=0$ and $\omega_0=0$), we have $\gamma_0=0$ and $M_0$ is the projector on $\theta$. The fact that the direction of oscillations has been shifted by $\eps^\beta \omega_0$ yields that all mass concentrating onto $\xi_0$ comes from the infinity with respect to the scale $\eps$. The contributions that we observe in Proposition~\ref{prop:disp1} are reminiscents of the measure $\gamma_t$ which happens to be non-zero in this situation. Since $V=0$, it is possible to calculate everything explicitely and one gets the following description, which implies Proposition~\ref{prop:disp1}:

\begin{lemma}
Assume  $\omega_0\in\ker \nabla^2\lambda(\xi_0)$, $\beta>{2\over 3}$ and $V=0$.
\begin{itemize}
\item[(i)] If $\alpha=0$, then  $\gamma_t(dx,d\omega) = \left| {\rm e}^{it \nabla^2\lambda(\xi_0)D_x\cdot D_x}\theta (x) \right|^2 dx  \delta_{\omega_0}(d\omega)$ and therefore:
$$\mu_t(dx,d\xi) = \left| {\rm e}^{it \nabla^2\lambda(\xi_0)D_x\cdot D_x}\theta (x) \right|^2 dx  \delta_{\xi_0}(d\xi).$$
\item[(ii)] If $\alpha\neq 0$, then $\gamma_t(dx,d\omega) = \|\theta\|_{L^2(\R^d)}^2 \delta_0(dx)  \delta_{\omega_0}(d\omega)$ and in particular:
$$\mu_t(dx,d\xi) =\|\theta\|_{L^2(\R^d)}^2 \delta_0(dx)  \delta_{\xi_0}(d\xi).$$
\end{itemize}
\end{lemma}

\begin{proof}
The proof relies on the analysis of the product
$$L^\eps := \left(\op_\eps\left(a\left(x,\xi,\frac{\xi-\xi_0}{\eps}\right)\right)u^\eps(t),u^\eps(t)\right),\quad a\in\mathcal{S}^0(d).$$
We observe that $L^\eps$ reads:
$$\displaylines{L^\eps = (2\pi)^{-3d}\eps^{-3d-d\alpha}\int_{\R^{7d}} a\left({x+y\over 2},\xi,{\xi-\xi_0\over \eps}\right)\overline \theta\left({x'\over \eps^\alpha}\right) \theta\left({y'\over\eps^\alpha}\right) \hfill\cr\hfill
\times\, {\rm Exp}\left[{i\over \eps}\left( (x-y)\cdot \xi -(x'-y')(\xi_0+\eps^\beta\omega_0)+\zeta\cdot(y-y')-\eta\cdot(x-x')\right)\right]\cr\hfill
\times\,{\rm Exp}\left[{it\over \eps^2} \left(\lambda(\eta)-\lambda(\zeta)\right) \right]
 dx'\,dy'\,dx\, dy\,d\xi\, d\zeta\,d\eta.\cr}$$
We perform the change of variables 
$$\displaylines{x=\eps^\alpha X,\;\;x'=\eps^\alpha X',\;\;
y=\eps^\alpha Y,\;\;y'=\eps^\alpha Y',\cr
\xi=\xi_0+\eps^\beta\omega_0+\eps^{1-\alpha}\xi',\;\; \zeta=\xi_0+\eps^\beta\omega_0+\eps^{1-\alpha}\zeta',
\;\;\eta=\xi_0+\eps^\beta\omega_0+\eps^{1-\alpha}\eta'
\cr}$$
in order to obtain
$$\displaylines{\qquad L^\eps = (2\pi)^{-3d}\int_{\R^{7d}} a\left(\eps^\alpha {X+Y\over 2},\xi_0+\eps^\beta\omega_0+\eps^{1-\alpha}\xi',\eps^{-1+\beta}(\omega_0+\eps^{1-\alpha-\beta}\xi')\right)\hfill\cr\hfill
\times\, {\rm Exp}\left[i\left( \xi'\cdot(X-Y)+ \zeta'\cdot(Y-Y')-\eta'\cdot(X-X')\right) +{it\over \eps^2} \Gamma_\eps(\zeta',\eta')\right] \cr\hfill
\times \, \overline \theta\left(X'\right) \theta\left(Y'\right) dX'\,dY'\,dX\, dY\,d\xi'\, d\zeta'\,d\eta',\qquad\cr}$$
with 
\begin{eqnarray*}
\Gamma_\eps(\zeta',\eta') & = & 
\lambda\left(\xi_0+\eps^\beta\omega_0+\eps^{1-\alpha}\zeta'\right)-\lambda\left(\xi_0+\eps^\beta\omega_0+\eps^{1-\alpha}\eta'\right)\\
& = & \eps^{2(1-\alpha)} \left(\nabla^2\lambda(\xi_0)\eta'\cdot\eta' - \nabla^2\lambda(\xi_0)\zeta'\cdot\zeta'\right) + O(\eps ^{3\beta}),
\end{eqnarray*}
where we have used $\nabla^2\lambda(\xi_0)\omega_0=0$ and $\beta < 1-\alpha$.
Since  $3\beta >2$, the term in $O(\eps^{3\beta})$ will be negligible in the phase. Now, the situation depends on whether $\alpha=0$ or not. 

\medskip

If $\alpha\not=0$, by use of Taylor expansion and by the definition of $a$, one easily convinces oneself that  
$$\displaylines{ L^\eps \sim a_\infty\left(0,\xi_0,\omega_0\right) (2\pi)^{-3d}\int_{\R^{7d}} \overline \theta\left(X'\right) \theta\left(Y'\right) 
{\rm Exp}\left[i\left( \xi'\cdot(X-Y)+ \zeta'\cdot(Y-Y')-\eta'\cdot(X-X')\right) \right]\cr\hfill 
\times\,{\rm Exp}\left[{it\over \eps^{2\alpha}} \left(\nabla^2\lambda(\xi_0)\eta'\cdot\eta'-\nabla^2\lambda(\xi_0)\zeta'\cdot\zeta'\right) \right] 
\, dX '\,dY'\,dX\, dY\,d\xi'\, d\zeta'\,d\eta'.\cr}$$
The integration in $\xi'$ generates a Dirac mass $\delta(X-Y)$, then the integration in $X$ generates a Dirac mass $\delta(\zeta'-\eta')$, whence 
$$L^\eps\sim (2\pi)^{-d}a_\infty\left(0,\xi_0,\omega_0\right)\int_{\R^{3d}}\overline \theta\left(X'\right) \theta\left(Y'\right)
{\rm Exp}\left[i\eta'\cdot(X'-Y') \right]
\, dX '\,dY'\,d\eta',$$
whence 
$$L^\eps \sim a_\infty\left(0,\xi_0,\omega_0\right)\| \theta\| _{L^2(\R^d)}.$$

If $\alpha =0$, similar arguments give 
$$\displaylines{ L^\eps \sim (2\pi)^{-3d}\int_{\R^{7d}}a_\infty\left({X+Y\over 2},\xi_0,\omega_0\right)  \overline \theta\left(X'\right) \theta\left(Y'\right)
{\rm Exp}\left[i\left( \xi'\cdot(X-Y)+ \zeta'\cdot(Y-Y')-\eta'\cdot(X-X')\right) \right]\cr\hfill 
\times\,{\rm Exp}\left[it\left(\nabla^2\lambda(\xi_0)\eta'\cdot\eta'-\nabla^2\lambda(\xi_0)\zeta'\cdot\zeta'\right) \right] 
\, dX '\,dY'\,dX\, dY\,d\xi'\, d\zeta'\,d\eta'.\cr}$$
Integration in $\xi'$ generates a Dirac mass $\delta(X-Y)$ and integration in $Y'$ and $X'$ give 
$$\displaylines{ L^\eps \sim (2\pi)^{-2d}\int_{\R^{3d}}a_\infty\left(X,\xi_0,\omega_0\right)  \overline{\widehat \theta}\left(\eta'\right) \widehat\theta\left(\zeta'\right)
{\rm Exp}\left[i X\cdot (\zeta'-\eta') \right]\hfill\cr\hfill 
\times\,{\rm Exp}\left[it\left(\nabla^2\lambda(\xi_0)\eta'-\nabla^2\lambda(\xi_0)\zeta'\cdot\zeta'\cdot\eta'\right) \right] 
\,dX\, d\zeta'\,d\eta'.\cr}$$
We deduce 
$$L^\eps \sim \int_{\R^d} a_\infty\left(x,\xi_0,\omega_0\right) | {\rm e}^{it \nabla^2\lambda(\xi_0)D_x\cdot D_x}\theta (x)|^2 dx,$$
as stated in the Proposition.  
\end{proof}


\section{Some comments on the case of a manifold of critical points} \label{sec:submanifold}

The proof of Theorem~\ref{theo:nondis} follows essentially the lines of that of Theorem \ref{theo:disc}; in particular a result analogous to Theorem \ref{theo:2micro} holds, based on the two-microlocal semiclassical measures described in Section \ref{sec:31}. For the proof of a more general result, the reader may consult~\cite{CFM}.

\medskip 

In this section, we develop the arguments of Remark~\ref{rem:critical}, showing that whenever $\dim\Lambda=p>0$, then the weak limit $\nu$ of the energy densities $|u^\eps(t,\cdot)|^2dx$ may not be absolutely continuous with respect to $dx$.

\medskip 

Let us first assume {\bf A0}, {\bf A1}, {\bf A2'} and suppose that the Hessian of $\lambda$ at its critical points is of maximal rank so that we can use Theorems~\ref{theo:nondis}, \ref{theo:simple}. Suppose that $V=0$ and $\Lambda =\{(\xi',0)\in\R^d\}$, where as before we write $\xi=(\xi',\xi'')$, with $\xi'\in\R^{r}$, $r=d-p$, and $\xi''\in\R^p$. We consider  initial data of the form 
$$u^\eps_0(x)=\theta\left({x''}\right)v^\eps(x'),$$
where $\alpha\in[0,1)$, $\|\theta\|_{L^2(\R^{p})}=1$, and $v^\eps$ is a uniformly bounded family of $L^2(\R^{r})$ admitting only one semiclassical measure $m(dx',d\xi')$. 

\medskip 

In view of the proof of  and the choice of the initial data, we have (with the notations of Theorem~\ref{theo:nondis}):
$$M_0(x',\xi')=|\theta\rangle \langle \theta | \qquad {\rm and} \qquad \nu^0(dx',d\xi',d\xi'')= m(dx',d\xi')\,  \delta_0(d\xi''),$$
whence, by Theorem~\ref{theo:nondis}: 
$$M_t(x',\xi') = | \theta(t,\xi',\cdot\,)\rangle \langle \theta(t,\xi',\cdot\,)|,\;\;{\rm with}\;\;
\theta(t,\xi',y) = {\rm e}^{-\frac{it}{2} \nabla^2_{\xi''} \lambda (\xi',0) D_y\cdot D_y }\theta (y).$$
We deduce: 
$$\mu_t(dx,d\xi) =|\theta(t,\xi',x'')|^2 dx'' \otimes m(dx',d\xi')\otimes  \delta_0(d\xi'').$$

One sees that, if the projection of $m$ on the position space is not absolutely continuous with respect to $dx'$, then the measure describing the weak limit of the energy density will also be singular. 

\begin{corollary}\label{cor:critical}
The choice of 
$$v^\eps(x')= \eps^{p-d\over 2} \varphi\left({x'-z_0\over \eps}\right) {\rm e}^{ {i\over \eps}x'\cdot \zeta_0},$$
which is the one of Remark~\ref{rem:critical}, implies 
$$m(x',\xi')= \| \varphi\|^2_{L^2(\R^{r})}\delta_{z_0}(dx')\otimes\delta_{\zeta_0}(d\xi'),$$
whence equation~(\ref{eq:critical}).
\end{corollary}

Of course, in the case where the Hessian of $\lambda$ is not of maximal rank on $\Lambda$, for example at a precise point $\xi_0=(\xi'_0,0)$, there happens a phenomenon similar to those described in Section~\ref{sec:deg}. Let us take a family $(v^\eps)_{\eps > 0}$ which oscillates along the vector~$(\xi'_0,0)$ as: 
$$v^\eps(x')={\rm e}^{{i\over\eps}\xi'_0\cdot x'} \varphi(x'),\;\;\varphi\in{\mathcal C}_0^\infty(\R^{r}).$$
Besides, as in section~\ref{sec:deg}, we add shifted oscillations in $x''$ along a vector $\omega_0\in{\bf S}^{p-1}$ by setting: 
$$u^\eps_0(x) =\eps^{-\frac{\alpha p}{2}}\theta\left({x''\over \eps^\alpha}\right) {\rm e}^{{i\over \eps^{1-\beta}}x''\cdot \omega_0} v^\eps(x').$$
The full picture is described in the following proposition. 

 \begin{proposition}
Assume $\nabla_{\xi''}\lambda(\xi'_0,0)\omega_0=0$ and $\beta>{2\over 3}$.
\begin{itemize}
\item[(i)] If $\alpha=0$, then: 
$$\mu_t(dx,d\xi)= \left| \varphi(x') e^{it\left( \nabla^2_{\xi''}\lambda(\xi_0)D_{x''} \cdot D_{x''})  \right)}\theta(x'')\right|^2 dx \otimes \delta_{\xi_0}(d\xi).$$
\item[(ii)] If $\alpha\not=0$, then:
$$\mu_t(dx,d\xi)=|\varphi(x')|^2 dx' \otimes \delta_0(dx'') \otimes \delta_{\xi_0} (d\xi).$$
\end{itemize}
\end{proposition}

Here again, we see that both situations may occur when the Hessian is not of maximal rank: absolute continuity with respect to Lebesgue measure or singularity.

\begin{proof}
The proof relies on the analysis of the integral 
\begin{eqnarray*}
L^\eps & = & \left(\op_\eps\left(a\left(x,\xi,{\xi''\over\eps}\right)\right)u^\eps(t),u^\eps(t)\right)\\
& = & (2\pi\eps)^{-3d}  \eps^{-p\alpha}\int_{\R^{7d}} a\left({x+y\over 2},\xi,{\xi''\over \eps}\right)
{\rm Exp}\left[  {i\over \eps}\left(\xi\cdot(x-y)+\zeta\cdot(y-z)-\eta\cdot(x-r)+\xi'_0\cdot (z'-r')\right) \right]\\
&&\times\, {\rm Exp}\left[{i\over\eps^{1-\beta}}\omega_0\cdot(z''-r'')\right]{\rm Exp}\left[ {it\over \eps^2}\left(\lambda(\eta)-\lambda(\zeta)\right)\right] \varphi(z') \overline\varphi(r') \theta\left({z''\over \eps^\alpha}\right) \overline\theta\left({r''\over \eps^\alpha}\right) d\zeta\,dz\,d\eta\, dr\,d\xi\, dx\, dy.
\end{eqnarray*}
 We perform the change of variables 
$$\displaylines{
\widetilde x''=\eps^\alpha x'',\;\;\widetilde y''=\eps^\alpha y'',\;\;\widetilde z''=\eps^\alpha z'',\;\;\widetilde r''=\eps^\alpha r'',\cr
\widetilde\xi''=\eps^\beta \omega_0 +\eps^{1-\alpha}\xi'',\;\;\widetilde\zeta''=\eps^\beta \omega_0 +\eps^{1-\alpha}\zeta'',\;\;\widetilde\eta''=\eps^\beta \omega_0 +\eps^{1-\alpha}\eta'',\cr
\widetilde \xi'=\xi'_0+\eps\xi',\;\;\widetilde \zeta'=\xi'_0+\eps\zeta',\;\;\widetilde \eta'=\xi'_0+\eps\eta',
}$$
and obtain (letting the tildas down): 
$$\displaylines{
L^\eps  =  (2\pi)^{-3d} \int_{\R^{7d}} a\left({x'+y'\over 2},\eps^\alpha {x''+y''\over 2},\xi'_0+\eps\xi',\eps^{1-\alpha}\xi''+\eps^\beta\omega_0,\eps^{\beta -1}(\omega_0+\eps^{1-\alpha-\beta}\xi'')\right)\hfill\cr\hfill
\times\, {\rm Exp}\left[  i\xi\cdot(x-y)+i\zeta\cdot(y-z)-i\eta\cdot(x-r) \right]
{\rm Exp}\left[ {it\over \eps^2}\Gamma_\eps(\zeta,\eta)\right] \varphi(z') \overline\varphi(r') \theta\left({z''}\right) \overline\theta\left({r''}\right) d\zeta\,dz\,d\eta\, dr\,d\xi\, dx\, dy,
\cr}$$
where, using the assumptions on $\xi'_0$ and $\omega_0$,
\begin{eqnarray*}
\Gamma_\eps(\zeta,\eta)& = & \lambda_j (\xi'_0+\eps\eta',\eps^\beta\omega_0+\eps^{1-\alpha}\eta'') - \lambda(\xi'_0+\eps\zeta',\eps^\beta\omega_0+\eps^{1-\alpha}\zeta'') \\
&= &\eps^{2(1-\alpha)} \left( \nabla^2_{\xi''}\lambda(\xi'_0,0)(\eta'',\eta'') - \nabla^2_{\xi''}\lambda(\xi'_0,0)(\zeta'',\zeta'') \right) +O(\eps^{3\beta}).
\end{eqnarray*}
As a consequence, if $\alpha=0$, 
$$\displaylines{
L^\eps  \sim  (2\pi)^{-3d} \int_{\R^{7d}} a_\infty\left({x+y\over 2},\xi'_0,0,\frac{\omega_0}{\|\omega_0\|}\right)
{\rm Exp}\left[  i\xi\cdot(x-y)+i\zeta\cdot(y-z)-i\eta\cdot(x-r) \right]\hfill\cr\hfill
\times\, {\rm Exp}\left[ it\left( \nabla^2_{\xi''}\lambda(\xi'_0,0)(\eta'',\eta'') - \nabla^2_{\xi''}\lambda(\xi'_0,0)(\zeta'',\zeta'') \right)\right]
 \varphi(z') \overline\varphi(r') \theta\left({z''}\right) \overline\theta\left({r''}\right) d\zeta\,dz\,d\eta\, dr\,d\xi\, dx\, dy.
\cr}$$
Integration in $\xi$ generates a Dirac mass $\delta(x-y)$, whence 
\begin{eqnarray*}
L^\eps & \sim & (2\pi)^{-2d} \int_{\R^{5d}} a_\infty\left(x,\xi'_0,0,\frac{\omega_0}{\|\omega_0\|}\right)
 {\rm Exp}\left[   i\zeta \cdot (x-z)-i\eta\cdot (x-r) \right]\\
&&\times\, {\rm Exp}\left[  it\left( \nabla^2_{\xi''}\lambda(\xi'_0,0)(\eta'',\eta'') - \nabla^2_{\xi''}\lambda(\xi'_0,0)(\zeta'',\zeta'') \right)\right]
 \varphi(z') \overline\varphi(r') \theta\left({z''}\right) \overline\theta\left({r''}\right) d\zeta\,dz\,d\eta\, dr\, dx\\
& \sim &  \int_{\R^d} a_\infty\left(x,\xi'_0,0,\frac{\omega_0}{\|\omega_0\|}\right)
|\varphi(x')|^2\left| {\rm Exp}\left[- it\left( \nabla^2_{\xi',\xi'}\lambda(\xi'_0,0)(\nabla_{x''},\nabla_{x''})  \right)\right]\theta(x'')\right| ^2dx.
\end{eqnarray*}
Similarly, when $\alpha\not=0$,
\begin{eqnarray*}
L^\eps & \sim & (2\pi)^{-3d} \int_{\R^{7d}} a_\infty\left({x'+y'\over 2},0,\xi'_0,0,\frac{\omega_0}{\|\omega_0\|}\right)
{\rm Exp}\left[ \frac{it}{\eps^{2\alpha}}\left( \nabla^2_{\xi''}\lambda(\xi'_0,0)(\eta'',\eta'') - \nabla^2_{\xi''}\lambda(\xi'_0,0)(\zeta'',\zeta'') \right)\right]
\\
&&\times\, 
{\rm Exp}\left[  i\xi\cdot(x-y)+i\zeta\cdot(y-z)-i\eta\cdot(x-r) \right]
\varphi(z') \overline\varphi(r') \theta\left({z''}\right) \overline\theta\left({r''}\right) d\zeta\,dz\,d\eta\, dr\,d\xi\, dx\, dy.
\end{eqnarray*}
Integration in $\xi$ generates a Dirac mass $\delta(x-y)$, then integration in $x''$ generates a Dirac mass $\delta(\zeta''-\eta'')$ and we obtain 
\begin{eqnarray*}
L^\eps & \sim & (2\pi)^{-2d+p} \int_{\R^{5d-2p}} a_\infty\left(x',0,\xi'_0,0,\frac{\omega_0}{\|\omega_0\|}\right)
{\rm Exp}\left[  i\zeta'(x'-z')- i\eta'\cdot(x'-r') -i\zeta''(z''-r'')\right]\\
&& \qquad \times\,\varphi(z') \overline\varphi(r') \theta\left({z''}\right) \overline\theta\left({r''}\right) d\zeta'\,dz\,d\eta'\, dr\, dx'\\
&\sim & \| \theta\|^2_{L^2(\R^p)} \int_{\R^{d-p}} a_\infty\left(x',0,\xi'_0,0,\frac{\omega_0}{\|\omega_0\|}\right) |\varphi(x')|^2 dx'.
\end{eqnarray*}
\end{proof}


\end{document}